\documentclass[a4paper,11pt]{amsart}
\usepackage{amsmath}
\usepackage{amstext}
\usepackage{amsbsy}
\usepackage{amsopn}
\usepackage{upref}
\usepackage{amsthm}
\usepackage{amsfonts}
\usepackage{amssymb}
\usepackage{mathrsfs}
\allowdisplaybreaks
 \newtheorem{theorem}{Theorem}

 \newtheorem{lemma}[theorem]{Lemma}
\theoremstyle{definition}
 
\theoremstyle{remark}
 
\begin{document}
%%%%%%
%%%%%% Make Title
%%%%%%
\title[Berezin-Toeplitz operators]{Bounded Berezin-Toeplitz operators \\ on the Segal-Bargmann space}
\author[H.~Chihara]{Hiroyuki Chihara}
\address{Mathematical Institute,  
         Tohoku University, 
         Sendai 980-8578, Japan}
\email{chihara@math.tohoku.ac.jp}
\begin{abstract}
We discuss 
the boundedness of Berezin-Toeplitz operators 
on a generalized Segal-Bargmann space (Fock space) 
over the complex $n$-space. 
This space is characterized 
by the image of a global Bargmann-type transform 
introduced by Sj\"ostrand. 
We also obtain the deformation estimates of 
the composition of Berezin-Toeplitz operators 
whose symbols and their derivatives up to order three 
are in the Wiener algebra of Sj\"ostrand. 
Our method of proofs is based 
on the pseudodifferential calculus 
and the heat flow determined by  
the phase function of the Bargmann transform. 
\end{abstract}
\thanks{Supported by the JSPS Grant-in-Aid for Scientific Research \#20540151.} 
\maketitle
%%%%%%
%%%%%% Mainmatter
%%%%%%
%
%%%%%%
%%%%%% Section 1
%%%%%%
\section{Introduction}
\label{section:introduction}
We study the boundedness and the deformation estimates of
Berezin-Toeplitz operators on a generalized Segal-Bargmann space (Fock
space) introduced by Sj\"ostrand in \cite{sjoestrand3}. 
This space is a reproducing kernel Hilbert space of 
square-integrable holomorphic functions on
the complex $n$-space, and is characterized 
by the image of a global Bargmann-type transform. 
We begin with a review of Sj\"ostrand's ``linear'' theory in
\cite{sjoestrand3} to introduce the setting of the present paper. 
Let $\phi(X,Y)$ be a quadratic form of
$(X,Y)\in\mathbb{C}^n\times\mathbb{C}^n$ 
of the form
$$
\phi(X,Y)
=
\frac{1}{2}\langle{X,AX}\rangle
+
\langle{X,BY}\rangle
+
\frac{1}{2}\langle{Y,CY}\rangle,
$$
where $A$, $B$ and $C$ are complex $n\times{n}$ matrices, 
${}^tA=A$, ${}^tC=C$ 
and 
$\langle{X,Y}\rangle=X_1Y_1+\dotsb+X_nY_n$ 
for 
$X=(X_1,\dotsc,X_n)$ and $Y=(Y_1,\dotsc,Y_n)$. 
Set $i=\sqrt{-1}$, $C_R=(C+\bar{C})/2$ and $C_I=(C-\bar{C})/{2i}$. 
We denote by $I_n$ the $n{\times}n$ identity matrix. 
Assume that 
\begin{align}
  \det\phi^{\prime\prime}_{XY}
& =
  \det{B}\ne0,
\label{equation:phi1}
\\
  \operatorname{Im}\hspace{1pt}
  \phi^{\prime\prime}_{YY}
& =
  C_I>0.
\label{equation:phi2}
\end{align}
We remark that 
$\det{C}=\det{C_I}\det(C_I^{-1/2}C_RC_I^{-1/2}+iI_n)\ne0$ 
since $C_I^{-1/2}C_RC_I^{-1/2}$ is a real symmetric matrix. 
Let $h\in(0,1]$ be a semiclassical parameter, 
and let $\mathscr{S}(\mathbb{R}^n)$ be the Schwartz class on $\mathbb{R}^n$. 
A global Bargmann-type transformation of
$u\in\mathscr{S}(\mathbb{R}^n)$ is defined by 
$$
Tu(X)
=
C_\phi
h^{-3n/4}
\int_{\mathbb{R}^n}
e^{i\phi(X,y)/h}
u(y)
dy,
$$
where $C_\phi$ is a normalizing constant as
$$
C_\phi
=
2^{-n/2}\pi^{-3n/4}
\lvert\det{B}\rvert
(\det{C_I})^{-1/4}.
$$
\par
The assumption \eqref{equation:phi2} guarantees the existence of 
a function 
\begin{align*}
  \Phi(X)
& =
  \max_{y\in\mathbb{R}^n}
  \{-\operatorname{Im}\hspace{1pt}\phi(X,y)\}
\\
& =
  \frac{1}{2}
  \langle
  \operatorname{Im}\hspace{1pt}({}^tBX),C_I^{-1}\operatorname{Im}\hspace{1pt}({}^tBX)
  \rangle
  -
  \frac{1}{2}
  \operatorname{Im}\hspace{1pt}\langle{X,AX}\rangle
\\
& =
  \langle{X,\Phi^{\prime\prime}_{X\bar{X}}\bar{X}}\rangle
  +
  \operatorname{Re}\hspace{1pt}\langle{X,\Phi^{\prime\prime}_{XX}X}\rangle,
\end{align*}
$$
\Phi^{\prime\prime}_{X\bar{X}}
=
\frac{BC_I^{-1}{}^t\bar{B}}{4}>0, 
\quad
\Phi^{\prime\prime}_{XX}
=
-
\frac{BC_I^{-1}{}^tB}{4}-\frac{A}{2i}.
$$
We denote the Lebesgue measure on $\mathbb{C}^n$ by $L$. 
Set $\lvert{X}\rvert=\sqrt{\langle{X,\bar{X}}\rangle}$ 
for $X\in\mathbb{C}^n$. 
Let $L^2_\Phi$ be the set of all square-integrable functions on
$\mathbb{C}^n$ with respect to $e^{-2\Phi(X)/h}L(dX)$, 
and let $H_\Phi$ be the set of all holomorphic functions in $L^2_\Phi$. 
We remark that 
$$
\operatorname{Re}\hspace{1pt}\{i\phi(X,y)\}
=
\Phi(X)
-
\frac{1}{2}
\lvert
C_I^{1/2}
(y+C_I^{-1}\operatorname{Im}\hspace{1pt}({}^tBX))
\rvert^2. 
$$
The Bargmann transform $T$ is well-defined for any tempered distribution 
$u\in\mathscr{S}^\prime(\mathbb{R}^n)$. 
Moreover $Tu$ satisfies 
$e^{-\Phi(X)/h}Tu(X)\in\mathscr{S}^\prime(\mathbb{C}^n)$, 
and is holomorphic on $\mathbb{C}^n$. 
In particular, 
$T$ gives a Hilbert space isomorphism 
of $L^2(\mathbb{R}^n)$ onto $H_\Phi$, 
where $L^2(\mathbb{R}^n)$ is the set of 
all Lebesgue square-integrable functions on $\mathbb{R}^n$. 
We here remark that 
$e^{-\Phi(X)/h}T(\mathscr{S}(\mathbb{R}^n))\subset\mathscr{S}(\mathbb{C}^n)$, 
and $T(\mathscr{S}(\mathbb{R}^n))$ is densely embedded in 
$H_\Phi$ and $T(\mathscr{S}^\prime(\mathbb{R}^n))$ respectively 
since $\mathscr{S}(\mathbb{R}^n)$ is densely embedded in 
$L^2(\mathbb{R}^n)$ and $\mathscr{S}^\prime(\mathbb{R}^n)$ respectively. 
The Bargmann transform $T$ is interpreted as a 
Fourier integral operator associated with a linear canonical transform 
$$
\kappa_T:
\mathbb{C}^n\times\mathbb{C}^n
\ni
(Y,-\phi^\prime_Y(X,Y))
\mapsto
(X,\phi^\prime_X(X,Y))
\in
\mathbb{C}^n\times\mathbb{C}^n,
$$
$$
\kappa_T(x,\xi)=(-{}^tB^{-1}(Cx+\xi),Bx-A{}^tB^{-1}(Cx+\xi)).
$$
If we set
$$
\Lambda_\Phi
=
\left\{
\left(X,\frac{2}{i}\frac{\partial\Phi}{\partial{X}}(X)\right)
\ \Bigg\vert\ 
X\in\mathbb{C}^n
\right\}, 
$$
then $\Lambda_\Phi=\kappa_T(\mathbb{R}^{2n})$. 
This means that the singularities of $u\in\mathscr{S}^\prime(\mathbb{R}^n)$ 
described in the phase space $\mathbb{R}^{2n}$ 
are translated into those of $Tu$ 
described in the Lagrangian submanifold $\Lambda_\Phi$. 
\par
Let $\Psi(X,Y)$ be a holomorphic quadratic function on 
$\mathbb{C}^n\times\mathbb{C}^n$ defined by 
the critical value of 
$-\{\phi(X,Z)-\overline{\phi(\bar{Y},\bar{Z})}\}/2i$ 
for $Z\in\mathbb{C}^n$, that is, 
$$
\Psi(X,Y)
=
\langle{X,\Phi^{\prime\prime}_{X\bar{X}}Y}\rangle
+
\frac{1}{2}\langle{X,\Phi^{\prime\prime}_{XX}X}\rangle
+
\frac{1}{2}\langle{Y,\overline{\Phi^{\prime\prime}_{XX}}Y}\rangle.
$$
Note that $\Psi(X,\bar{X})=\Phi(X)$. 
$TT^\ast$ is an orthogonal projector of $L^2_\Phi$ onto $H_\Phi$, 
and given by 
\begin{equation}
TT^\ast{u}(X)
=
\frac{C_\Phi}{h^n}
\int_{\mathbb{C}^n}
e^{[2\Psi(X,\bar{Y})-2\Phi(Y)]/h}
u(Y)L(dY),
\label{equation:projector} 
\end{equation}
$$
C_\Phi
=
\left(\frac{2}{\pi}\right)^n
\det(\Phi^{\prime\prime}_{X\bar{X}})
=
(2\pi)^{-n}\lvert\det{B}\rvert^2(\det{C_I})^{-1}.
$$
\par
Here we state the definition of Berezin-Toeplitz operators on $H_\Phi$. 
If we set $R=C_I^{-1/2}{}^tB/2$, then $R^\ast{R}=\Phi^{\prime\prime}_{\bar{X}X}$. 
Let $\mathscr{T}$ be a class of symbols defined by 
$$
\mathscr{T}
=
\left\{
b(X)
\ \Bigg\vert\ 
\int_{\mathbb{C}^n}
e^{-2\lvert{R(X-Y)}\rvert^2/h}
\lvert{b(Y)}\rvert^2
L(dY)
<\infty
\ \text{for any}\ 
X\in\mathbb{C}^n
\right\}.
$$
A Berezin-Toeplitz operator $\tilde{T}_b$ associated with a symbol 
$b\in\mathscr{T}$ is defined by $\tilde{T}_bu=TT^\ast(bu)$ for $u{\in}H_\Phi$. 
Since 
$$
\operatorname{Re}\hspace{1pt}
\{2\Psi(X,\bar{Y})-2\Phi(Y)\}
=
\Phi(X)-\Phi(Y)-\lvert{R(X-Y)}\rvert^2,
$$
$e^{-\Phi(X)/h}\tilde{T}_bu(X)$ takes a finite value for each $X\in\mathbb{C}^n$ 
provided that $u{\in}L^2_\Phi$ and $b\in\mathscr{T}$. 
Historically, Berezin introduced this type of operators acting 
on a class of holomorphic functions over some complex spaces or manifolds, 
and established the foundation of geometric quantization 
in his celebrated paper \cite{berezin}. 
Properties of such operators and related problems 
on the usual Segal-Bargmann space have been investigated 
in several papers. 
See 
\cite{BC1}, 
\cite{BC2}, 
\cite{coburn1}, 
\cite{coburn2}, 
\cite{coburn3}, 
\cite{stroethoff} 
and references therein. 
\par
Here we give two examples of $H_\Phi$. 
\begin{description}
\item[{\bf Example 1}] 
If 
$\phi(X,Y)=i\beta(X^2/2-2XY+Y^2)$, 
$\beta>0$ 
and 
$XY=\langle{X,Y}\rangle$, 
then $H_\Phi$ is the usual Segal-Bargmann space (the Fock space), 
and 
$$
\Psi(X,\bar{Y})=\frac{\beta}{2}X\bar{Y}, 
\quad
\kappa_T(x,\xi)
=
\left(
x-\frac{i}{2\beta}\xi,
-i\beta\left(x+\frac{i}{2\beta}\xi\right)
\right).
$$
It is remarkable that $\Phi(X)=\beta\lvert{X}\rvert^2/2$ is strictly convex and 
$\Phi^{\prime\prime}_{XX}=0$ in this case. 
The strict convexity justifies 
the change of quantization parameter. 
See \cite[Proposition~1.3]{sjoestrand3}. 
These facts are effectively used in the analysis on the usual Segal-Bargmann space. 
See e.g., \cite{folland} for the detail. 
\item[{\bf Example 2}] 
If we set $\phi(X,Y)=i(X-Y)^2/2$, 
then $T$ is the heat kernel transform, and 
$$
\Psi(X,\bar{Y})
=
-\frac{(X-\bar{Y})^2}{8}, 
\quad
\Phi(X)
=
\frac{(\operatorname{Im}\hspace{1pt}X)^2}{2}, 
\quad
\kappa_T(x,\xi)=(x-i\xi,\xi).
$$
In this case, the global FBI transform $e^{-\Phi(X)/h}T$ 
and the space $H_\Phi$ are used as strong tools for 
microlocal and semiclassical analysis of 
linear differential operators on $\mathbb{R}^n$. 
See \cite{martinez} for the detail. 
\end{description}
\par
The purpose of the present paper is 
to study the boundedness and the deformation estimates 
of Berezin-Toeplitz operators on the generalized 
Segal-Bargmann space $H_\Phi$. 
To state our results, we introduce notation and 
review pseudodifferential calculus on $H_\Phi$ developed in \cite{sjoestrand3}. 
\par
We denote by $\mathscr{L}(H_\Phi)$ the set of 
all bounded linear operators of $H_\Phi$ to $H_\Phi$, 
and set 
$$
Q(a,b)
=
\left\langle
\frac{\partial{a}}{\partial{X}},
(\Phi^{\prime\prime}_{\bar{X}X})^{-1}
\frac{\partial{b}}{\partial\bar{X}}
\right\rangle,
\quad
\{a,b\}
=
iQ(a,b)-iQ(b,a)
$$
for $a,b{\in}C^1(\mathbb{C}^n)$. 
Pick up $\chi\in\mathscr{S}(\mathbb{C}^n)$ such that 
$\int_{\mathbb{C}^n}\chi(X)L(dX)\ne0$. 
Sj\"ostrand's Wiener algebra $S_\text{W}(\mathbb{C}^n)$ 
is the set of all tempered distributions on $\mathbb{C}^n$ satisfying 
\begin{equation}
U(\zeta;b)
=
\sup_{Z\in\mathbb{C}^n}
\lvert
\mathscr{F}[u\tau_Z\chi](\zeta)
\rvert
\in
L^1(\mathbb{C}^n_\zeta), 
\label{equation:sjoestrand}
\end{equation}
where $\mathscr{F}$ is the usual (not semiclassical) 
Fourier transform on $\mathbb{C}^n\simeq\mathbb{R}^{2n}$, 
$\tau_Z\chi(X)=\chi(X-Z)$, 
and $L^1(\mathbb{C}^n)$ is the set of 
all Lebesgue integrable functions on $\mathbb{C}^n$. 
Set 
$\lVert{b}\rVert_{S_\text{W}}=\lVert{U(\cdot;b)}\rVert_{L^1(\mathbb{C}^n)}$. 
We also denote by $L^\infty(\mathbb{C}^n)$ 
the set of all essentially bounded functions on $\mathbb{C}^n$. 
The definition of $S_\text{W}(\mathbb{C}^n)$ is independent of the
choice of $\chi$, 
and $S_\text{W}(\mathbb{C}^n)$ is invariant 
under linear transforms on $\mathbb{C}^n$. 
It is remarkable that 
$$
\mathscr{B}^{2n+1}(\mathbb{C}^n)
\subset
S_\text{W}(\mathbb{C}^n) 
\subset
\mathscr{B}^0(\mathbb{C}^n),
$$
and the Weyl quantization of any element of $S_\text{W}$ 
is a bounded linear operator. 
Set $\mathbb{N}_0=\{0,1,2,\dotsc\}$ for short. 
$\mathscr{B}^k(\mathbb{C}^n)$, $k\in\mathbb{N}_0$ 
is the set of all bounded $C^k$-functions on $\mathbb{C}^n$ 
whose derivatives of any order up to $k$ are also bounded on $\mathbb{C}^n$. 
\par
Next we introduce the Weyl quantization on $H_\Phi$. 
For fixed $X\in\mathbb{C}^n$, set 
$$
\Gamma(X)
=
\left\{
(Y,\theta) 
\ \Bigg\vert\ 
Y\in\mathbb{C}^n, \ 
\theta
=
\frac{2}{i}\frac{\partial\Phi}{\partial{X}}
\left(\frac{X+Y}{2}\right)
\right\}, 
$$
and a volume of $\Gamma(X)$ is defined by  
$d\Omega
=
dY_1\wedge\dotsb\wedge{dY_n}\wedge{d\theta_1}\wedge\dotsb\wedge{d\theta_n}$. 
For $u{\in}H_\Phi$, 
the reproducing formula $u=TT^\ast{u}$ has another expression
\begin{equation}
u(X)
=
\frac{1}{(2\pi{h})^n}
\int_{\Gamma(X)}
e^{i\langle{X-Y,\theta}\rangle/h}
u(Y)d\Omega. 
\label{equation:reproducer}
\end{equation}
The right hand sides of 
\eqref{equation:projector} and \eqref{equation:reproducer} 
coincide to each other via the change of variables 
called the Kuranishi trick. 
The Weyl quantization of a symbol 
$a(X,\theta){\in}S_\text{W}(\Lambda_\Phi)
=(\kappa_T^{-1})^\ast{S_{\text{W}}(\mathbb{R}^{2n})}$ 
is defined by 
$$
\operatorname{Op}_h^\text{W}(a)u(X)
=
\frac{1}{(2\pi{h})^n}
\int_{\Gamma(X)}
e^{i\langle{X-Y,\theta}\rangle/h}
a\left(\frac{X+Y}{2},\theta\right)
u(Y)
d\Omega
$$
for $u{\in}T(\mathscr{S}(\mathbb{R}^n))$. 
$\operatorname{Op}_h^\text{W}(a)u$ is holomorphic in $\mathbb{C}^n$ since 
$$
\frac{\partial}{\partial\bar{X}}
e^{i\langle{X-Y,\theta}\rangle/h}
a\left(\frac{X+Y}{2},\theta\right)
=
\frac{\partial}{\partial\bar{Y}}
e^{i\langle{X-Y,\theta}\rangle/h}
a\left(\frac{X+Y}{2},\theta\right)
$$
in the sense of distribution. 
The Weyl quantization of $a\circ\kappa_T$ 
is defined by 
$$
\operatorname{Op}_h^\text{W}(a\circ\kappa_T)u(x)
=
\frac{1}{(2\pi{h})^n}
\int_{\mathbb{R}^{2n}}
e^{i\langle{x-y,\xi}\rangle/h}
a\circ\kappa\left(\frac{x+y}{2},\xi\right)
u(y)
dyd\xi 
$$
for $u\in\mathscr{S}(\mathbb{R}^n)$. 
It is remarkable that 
$\operatorname{Op}_h^\text{W}(S_\text{W}(\Lambda_\Phi))$ 
is extended on $H_\Phi$ and a subalgebra of $\mathscr{L}(H_\Phi)$, 
and the exact Egorov theorem 
\begin{equation}
\operatorname{Op}_h^\text{W}(a)\circ{T}
=
T\circ\operatorname{Op}_h^\text{W}(a\circ\kappa_T) 
\label{equation:egorov}
\end{equation}
holds for $a{\in}S_\text{W}(\Lambda_\Phi)$. 
Moreover, Guillemin discovered in \cite{guillemin} that 
$\tilde{T}_b=\operatorname{Op}_h^\text{W}(b^\prime_{1/2})$ 
for $b(X)=b(X,\bar{X})$, where 
$$
b^\prime_{1/2}(X,\theta)
=
b_{1/2}
\left(
X,
(\Phi^{\prime\prime}_{X\bar{X}})^{-1}
\left(\frac{i}{2}\theta-\Phi^{\prime\prime}_{XX}X\right)
\right),
$$
and $\{b_t\}_{t\geqslant0}$ is the heat flow of $b$ defined by 
\begin{align*}
  b_t(X)
& =
  e^{th\Delta}b(X)
\\
& =
  \frac{C_\Phi}{(th)^n}
  \int_{\mathbb{C}^n}
  e^{-2\lvert{R(X-Y)}\rvert^2/th}
  b(Y)L(dY),
\end{align*}
$$
\Delta
=
\frac{1}{2}
\left\langle
\frac{\partial}{\partial{X}},
(\Phi^{\prime\prime}_{\bar{X}X})^{-1}
\frac{\partial}{\partial\bar{X}}
\right\rangle. 
$$
$b_t$ makes sense for $b\in\mathscr{T}$ and $t\in(0,2)$. 
We use only $t\in[0,1]$ as a quantization parameter. 
$b_1$ is said to be the Berezin symbol of 
a Berezin-Toeplitz operator $\tilde{T}_b$. 
These facts show that pseudodifferential calculus 
(See e.g., 
\cite{kumano-go}, \cite{martinez} and \cite{shubin}) 
and the heat flow determined by the phase function 
play essential roles in the analysis of 
Berezin-Toeplitz operators. 
\par
Here we state our results. 
\begin{theorem}
\label{theorem:bounded}
Suppose that $b\in\mathscr{T}$. We have 
\\
{\rm (i)}\ 
If $\tilde{T}_b\in\mathscr{L}(H_\Phi)$, then for any $t\in(1/2,1]$, 
\begin{equation}
\lVert{b_t}\rVert_{L^\infty(\mathbb{C}^n)}
\leqslant
\frac{\lVert\tilde{T}_b\rVert_{\mathscr{L}(H_\Phi)}}{(2t-1)^n}.
\label{equation:sakurako}
\end{equation}
{\rm (ii)}\ 
If $b_t{\in}L^\infty(\mathbb{C}^n)$ for some $t\in[0,1/2)$, then 
$\tilde{T}_b\in\mathscr{L}(H_\Phi)$. 
\\
{\rm (iii)}\ 
Suppose that $b\in\mathscr{S}^\prime(\mathbb{C}^n)$ in addition. 
Set 
$b^\lambda(X)
=e^{i\operatorname{Re}\hspace{0.5pt}\langle{X,\lambda}\rangle}
b(X)$ 
for 
$\lambda\in\mathbb{C}^n$. 
Then, 
$b_{1/2}{\in}S_\text{W}(\mathbb{C}^n)$ 
if and only if 
\begin{equation}
\lVert{(b^\lambda)_1(\cdot)}\rVert_{L^\infty(\mathbb{C}^n)}
e^{-h\lvert{}^tR^{-1}\lambda\rvert^2/8}
\in
L^1(\mathbb{C}^n_\lambda).
\label{equation:equivalent}
\end{equation}
In this case, $\tilde{T}_b\in\mathscr{L}(H_\Phi)$.
\end{theorem}
\begin{theorem}
\label{theorem:deformation}
Suppose that 
$\partial_X^\alpha\partial_{\bar{X}}^\beta{a}, 
\partial_X^\alpha\partial_{\bar{X}}^\beta{b} 
\in
S_\text{W}(\mathbb{C}^n)$ 
for any multi-indices satisfying 
$\lvert\alpha+\beta\rvert\leqslant3$. 
Then, there exists a positive constant $C_0$ 
which is independent of $a$, $b$ and $h$, such that   
\begin{align*}
& \left\lVert
  \tilde{T}_a\circ\tilde{T}_b
  -
  \tilde{T}_{ab}
  +
  \frac{h}{2}\tilde{T}_{Q(a,b)}
  \right\rVert_{\mathscr{L}(H_\Phi)},
  \quad
  \left\lVert
  [\tilde{T}_a,\tilde{T}_b]
  -
  \frac{ih}{2}\tilde{T}_{\{a,b\}}
  \right\rVert_{\mathscr{L}(H_\Phi)}
\\
& \qquad
  \leqslant
  C_0h^2
  \sum_{\lvert\alpha+\beta\rvert\leqslant3}
  \lVert\partial_X^\alpha\partial_{\bar{X}}^\beta{a}\rVert_{S_\text{W}}
  \sum_{\lvert\mu+\nu\rvert\leqslant3}
  \lVert\partial_X^\mu\partial_{\bar{X}}^\nu{b}\rVert_{S_\text{W}}.
\end{align*}
\end{theorem}
Here we explain the known results and the detail of our results. 
Theorem~\ref{theorem:bounded}-(i) is a refinement and a generalization of 
the results of Berger and Coburn in \cite{BC2}. 
They proved that 
$\lVert{b_t}\rVert_{L^\infty(\mathbb{C}^n)}\leqslant{C(t)}
\lVert\tilde{T}_b\rVert_{\mathscr{L}(H_\Phi)}$ 
for $t\in(1/2,1]$ with some function $C(t)$ 
in case that $H_\Phi$ is the usual Segal-Bargmann space. 
For a general $H_\Phi$, we need some ideas to avoid difficulties 
coming from $\Phi^{\prime\prime}_{XX}\ne0$. 
Theorem~\ref{theorem:bounded}-(ii) is obvious by the $L^2$-boundedness
theorem of pseudodifferential operators of order zero with smooth symbols. 
The condition \eqref{equation:equivalent} 
is a special form of the condition 
for which $b_{1/2}{\in}S_\text{W}(\mathbb{C}^n)$. 
This is given by a special choice of a Schwartz function 
$\chi$ appearing in the definition of $S_\text{W}(\mathbb{C}^n)$. 
Theorem~\ref{theorem:bounded}-(iii) seems to extend the known results 
by Berger and Coburn in \cite[Theorem~13]{BC2}, that is,  
if $b\geqslant0$ and $b_1{\in}L^\infty(\mathbb{C}^n)$, 
then $\tilde{T}_b\in\mathscr{L}(H_\Phi)$. 
\par
Theorem~\ref{theorem:deformation} reminds us of the recent interesting
results of Lerner and Morimoto in \cite{lm} on the Fefferman-Phong inequality. 
Coburn proved in \cite{coburn1} 
the deformation estimates on the usual Segal-Bargmann space 
under the assumption 
$$
a,b\in
\text{the set of all trigonometric polynomials}
+
C_0^{2n+6}(\mathbb{C}^{n}), 
$$
where 
$C_0^{2n+6}(\mathbb{C}^{n})$ 
is the set of all compactly supported $C^{2n+6}$-functions on $\mathbb{C}^n$. 
Roughly speaking, Theorem~\ref{theorem:deformation} asserts that 
the deformation estimates hold for $a,b\in\mathscr{B}^{2n+4}(\mathbb{C}^n)$. 
The relationship between Berezin-Toeplitz operators and Weyl 
pseudodifferential operators on $H_\Phi$ gives a formal identity 
$$
\tilde{T}_a\circ\tilde{T}_b
=\tilde{T}_c, 
\quad
c=e^{-h\Delta/2}(a^\prime_{1/2}{\#}b^\prime_{1/2}), 
$$
where $\#$ is the product of $S_\text{W}(\Lambda_\Phi)$ 
in the sense of the Weyl calculus introduced later.    
Unfortunately, however, the backward heat kernel $e^{-h\Delta/2}$ 
can act only on a class of real-analytic symbols, 
and it is very hard to obtain the symbol $c$. 
We apply the forward heat kernel $e^{th\Delta}$ 
to the construction of the asymptotic expansion 
of the backward heat kernel
$$
e^{-h\Delta/2}
=
1-\frac{h}{2}\Delta+\mathcal{O}(h^2), 
$$
and give an elementary proof of Theorem~\ref{theorem:deformation}.
\par
The organization of the present paper is as follows. 
In Section~\ref{section:bounded} 
we prove (i) and (iii) of Theorem~\ref{theorem:bounded}. 
In Section~\ref{section:deformation} we prove
Theorem~\ref{theorem:deformation}. 
%%%%%%
%%%%%% Section 2
%%%%%%
\section{Boundedness of Berezin-Toeplitz operators}
\label{section:bounded}
In this section we prove (i) and (iii) of
Theorem~\ref{theorem:bounded}. 
On one hand, to prove (i), we express the boundedness of $\tilde{T}_b$ 
in terms of a complete orthonormal system of $H_\Phi$. 
We introduce a trace class operator defined 
by $\tilde{T}_b$ and the complete orthonornal system, 
and take its trace which becomes $b_t(X)$ for any fixed
$X\in\mathbb{C}^n$. 
This idea is basically due to Berger and Coburn in \cite{BC2}. 
In our case, however, $\Phi(X)$ is not supposed to be strictly convex, 
nor $\Phi^{\prime\prime}_{XX}$ is not supposed to vanish. 
We need to be careful about these obstructions. 
On the other hand, the proof of (iii) is a simple computation. 
We choose a Schwartz function $\chi$ as a heat kernel at the time $t=1/2$.  
\par
Here we give two lemmas used in the proof of (i). 
For $u,v{\in}H_\Phi$, the inner product
$\langle\cdot,\cdot\rangle_{H_\Phi}$ is defined by 
$$
\langle{u,v}\rangle_{H_\Phi}
=
\int_{\mathbb{C}^n}
u(X)\overline{v(X)}
e^{-2\Phi(X)/h}
L(dX), 
$$
which is the restriction of 
$\langle\cdot,\cdot\rangle_{L^2_\Phi}$ on $H_\Phi$. 
Set 
$$
u_\alpha(X)
=
\left\{
\frac{C_\Phi}{h^n}
\frac{2^{\lvert\alpha\rvert}}{\alpha!h^{\lvert\alpha\rvert}}
\right\}^{1/2}
(RX)^\alpha
e^{\langle{X,\Phi^{\prime\prime}_{XX}X}\rangle/h}
$$
for a multi-index $\alpha\in\mathbb{N}_0^n$. 
The first lemma is concerned with a complete orthonormal system of $H_\Phi$ 
which is naturally generated by the Taylor expansion of 
the reproducing kernel $e^{2\Psi(X,\bar{Y})/h}$. 
\begin{lemma}
\label{theorem:cons} 
$\{u_\alpha\}_{\alpha\in\mathbb{N}_0^n}$ 
is a complete orthonormal system of $H_\Phi$. 
\end{lemma}
In case that $H_\Phi$ is the usual Segal-Bargmann space, 
the proof of Lemma~\ref{theorem:cons} is given in 
\cite[page 40, (1.63) Theorem]{folland}. 
In this case, $\{T^\ast{u_\alpha}\}_{\alpha\in\mathbb{N}_0^n}$ 
is said to be the family of Hermite functions.  
The general case can be proved in the same way, 
and we here omit the proof of Lemma~\ref{theorem:cons}. 
\par
Next lemma is concerned with the family of Weyl operators, 
which is a family of unitary operators on $H_\Phi$ 
and acts on symbols of Berezin-Toeplitz operators 
as a group of shifts on $\mathbb{C}^n$. 
The family of Weyl operators 
$\{W_\lambda\}_{\lambda\in\mathbb{C}^n}$ on $H_\Phi$ is defined by 
$$
W_\lambda{u(X)}
=
e^{[2\varphi(X,\lambda)-\varphi(\lambda,\lambda)]/h}u(X-\lambda),
$$
where 
$$
\varphi(X,\lambda)
=
\langle{X,\Phi^{\prime\prime}_{X\bar{X}}\bar{\lambda}}\rangle
+
\langle{X,\Phi^{\prime\prime}_{XX}\lambda}\rangle.
$$
We remark that $\varphi(X,\lambda)$ is holomorphic in $X$, and 
if $u$ is holomorphic, then $W_{\lambda}u$ is also. 
Properties of Weyl operators are the following. 
\begin{lemma}
\label{theorem:weyl}
We have 
\\
{\rm (i)}\ 
$W_\lambda^\ast=W_{-\lambda}$ on $H_\Phi$.
\\
{\rm (ii)}\ 
$W_\lambda^\ast{\circ}W_{\lambda}=I$ on $H_\Phi$.
\\
{\rm (iii)}\ 
$W_{\lambda}^\ast\circ\tilde{T}_b{\circ}W_{\lambda}
=\tilde{T}_{b(\cdot+\lambda)}$ 
on $H_\Phi$ for $b\in\mathscr{T}$.
\end{lemma}
\begin{proof}
A direct computation shows that 
\begin{align}
  2\varphi(X+\lambda,\lambda)
  -
  \varphi(\lambda,\lambda)
  -
  2\Phi(X+\lambda)
& =
  -
  2\overline{\varphi(X+\lambda,\lambda)}
  +
  \overline{\varphi(\lambda,\lambda)}
  -
  2\Phi(X)
\label{equation:reiko}
\\
& =
  2\overline{\varphi(X,-\lambda)}
  -
  \overline{\varphi(-\lambda,-\lambda)}
  -
  2\Phi(X).
\label{equation:nanako}
\end{align}
Let $u,v{\in}H_\Phi$. 
Using a translation $X\mapsto{X+\lambda}$ and \eqref{equation:nanako}, 
we deduce 
\begin{align*}
  \langle{W_\lambda{u},v}\rangle_{H_\Phi}
& =
  \int_{\mathbb{C}^n}
  e^{[2\varphi(X,\lambda)-\varphi(\lambda,\lambda)-2\Phi(X)]/h}
  u(X-\lambda)\overline{v(X)}
  L(dX)
\\
& =
  \int_{\mathbb{C}^n}
  e^{[2\varphi(X+\lambda,\lambda)-\varphi(\lambda,\lambda)-2\Phi(X+\lambda)]/h}
  u(X)\overline{v(X+\lambda)}
  L(dX)
\\
& =
  \int_{\mathbb{C}^n}
  e^{[2\overline{\varphi(X,-\lambda)}
     -
     \overline{\varphi(-\lambda,-\lambda)}
     -
     2\Phi(X)]/h}
  u(X)\overline{v(X+\lambda)}
  L(dX)
\\
& =
  \langle{u,{W_{-\lambda}}v}\rangle_{H_\Phi},
\end{align*}
which shows that $W_{\lambda}^\ast=W_{-\lambda}$. 
\par
$W_\lambda^\ast{\circ}W_\lambda=I$ is also proved by a direct computation
\begin{align*}
  W_\lambda^\ast{\circ}W_{\lambda}u(X)
& =
  W_{-\lambda}{\circ}W_{\lambda}u
\\
& =
  e^{[2\varphi(X,-\lambda)-\varphi(-\lambda,-\lambda)]/h}(W_{\lambda}u)(X+\lambda)
\\
& =
  e^{[-2\varphi(X,\lambda)-\varphi(\lambda,\lambda)]/h}(W_{\lambda}u)(X+\lambda)
\\
& =
  e^{[-2\varphi(X,\lambda)-\varphi(\lambda,\lambda)+2\varphi(X+\lambda,\lambda)-\varphi(\lambda,\lambda)]}u(X)=u(X),
\end{align*}
since $\varphi(X+\lambda,\lambda)=\varphi(X,\lambda)+\varphi(\lambda,\lambda)$. 
\par
$TT^\ast$ is self-adjoint on $L^2_\Phi$ and 
$TT^\ast{W_{\lambda}v}=W_{\lambda}v$ for $v{\in}H_\Phi$. 
Using this and \eqref{equation:reiko}, we deduce 
\begin{align*}
&
 \langle{W_\lambda^\ast\circ\tilde{T}_b{\circ}W_{\lambda}u,v}\rangle_{H_\Phi}
\\
  =
& \langle{\tilde{T}_b{\circ}W_{\lambda}u,W_{\lambda}v}\rangle_{H_\Phi}
\\
  =
& \langle{TT^\ast(bW_{\lambda}u),W_{\lambda}v}\rangle_{H_\Phi}
\\
  =
& \langle{bW_{\lambda}u,W_{\lambda}v}\rangle_{L^2_\Phi}
\\
  =
& \int_{\mathbb{C}^n}
  b(X)
  e^{[2\varphi(X,\lambda)
        +
        2\overline{\varphi(X,\lambda)}
        -
        \varphi(\lambda,\lambda)
        -
        \overline{\varphi(\lambda,\lambda)}
        -
        2\Phi(X)]/h}
\\
& \qquad\qquad\qquad
  \times
  u(X-\lambda)
  \overline{v(X-\lambda)} 
  L(dX)
\\
  =
& \int_{\mathbb{C}^n}
  b(X+\lambda)
  e^{[2\varphi(X+\lambda,\lambda)
      +
      2\overline{\varphi(X+\lambda,\lambda)}
      -
      \varphi(\lambda,\lambda)
      -
      \overline{\varphi(\lambda,\lambda)}
      -
      2\Phi(X+\lambda)]/h}
\\
& \qquad\qquad\qquad
  \times
  u(X)
  \overline{v(X)} 
  L(dX)
\\
  =
& \int_{\mathbb{C}^n}
  b(X+\lambda)
  e^{-2\Phi(X)/h}
  u(X)
  \overline{v(X)} 
  L(dX)
\\
  =
& \langle\tilde{T}_{b(\cdot+\lambda)}u,v\rangle_{H_\Phi},
\end{align*}
which proves 
$W_{\lambda}^\ast\circ\tilde{T}_b{\circ}W_{\lambda}=\tilde{T}_{b(\cdot+\lambda)}$.  
\end{proof}
Here we prove Theorem~\ref{theorem:bounded}-(i). 
\begin{proof}[Proof of Theorem~\ref{theorem:bounded}-{\rm (i)}] 
Suppose $\tilde{T}_{b}\in\mathscr{L}(H_\Phi)$, 
and set $M=\lVert\tilde{T}_b\rVert_{\mathscr{L}(H_\Phi)}$ for short. 
Lemma~\ref{theorem:weyl} shows that 
$\tilde{T}_{b(\cdot+X)}\in\mathscr{L}(H_\Phi)$ 
and $M=\lVert\tilde{T}_{b(\cdot+X)}\rVert_{\mathscr{L}(H_\Phi)}$ 
for any $X\in\mathbb{C}^n$. 
In terms of the complete orthonormal system 
given in Lemma~\ref{theorem:cons}, 
$\tilde{T}_b\in\mathscr{L}(H_\Phi)$ implies that 
$\lvert\langle{\tilde{T}_bu_\alpha,u_\beta}\rangle_{H_\Phi}\rvert\leqslant{M}$ 
for any $\alpha,\beta\in\mathbb{N}_0^n$. 
Since 
$\Phi(Y)
=\lvert{RY}\rvert^2
+\operatorname{Re}\hspace{1pt}\langle{Y,\Phi^{\prime\prime}_{XX}Y}\rangle$, 
we deduce that for any $X\in\mathbb{C}^n$
\begin{align*}
& \langle{\tilde{T}_{b(\cdot+X)}u_\alpha,u_\beta}\rangle_{H_\Phi}
\\
  =
& \langle{TT^\ast(b(\cdot+X)u_\alpha),u_\beta}\rangle_{H_\Phi}
\\
  =
& \langle{b(\cdot+X)u_\alpha,u_\beta}\rangle_{L^2_\Phi}
\\
  =
& \frac{C_\Phi}{h^n}
  \left(\frac{1}{\alpha!\beta!}\right)^{1/2}
  \int_{\mathbb{C}^n}
  b(X+Y)
\\
& \times
  \left\{\left(\frac{2}{h}\right)^{1/2}RY\right\}^\alpha
  \left\{\left(\frac{2}{h}\right)^{1/2}\overline{RY}\right\}^\beta
  e^{-2\lvert{RX}\rvert^2/h}
  L(dY).
\end{align*}
In particular, if we take $\alpha=\beta$ 
and sum it up for $\lvert\alpha\rvert=k$,  
then we have 
\begin{align}
& \sum_{\lvert\alpha\rvert=k}
  \langle{\tilde{T}_{b(\cdot+X)}u_\alpha,u_\alpha}\rangle_{H_\Phi}
\nonumber
\\
  =
& \frac{C_\Phi}{h^n}
  \int_{\mathbb{C}^n}
  \frac{1}{k!}
  \left(\frac{2\lvert{RY}\rvert^2}{h}\right)^k
  e^{-2\lvert{RY}\rvert^2/h}
  b(X+Y)
  L(dY).
\label{equation:detroit}
\end{align}
\par
Fix $(t,X)\in(1/2,1]\times\mathbb{C}^n$. 
When $k=0$, \eqref{equation:detroit} shows that 
$\langle{\tilde{T}_{b(\cdot+X)}u_0,u_0}\rangle_{H_\Phi}=b_1(X)$, 
and 
$\lvert\langle{\tilde{T}_{b(\cdot+X)}u_0,u_0}\rangle_{H_\Phi}\rvert\leqslant{M}$ 
implies that $\lVert{b_1}\rVert_{L^\infty(\mathbb{C}^n)}\leqslant{M}$, 
which is \eqref{equation:sakurako} at $t=1$.  
We consider $(t,X)\in(1/2,1)\times\mathbb{C}^n$ below, 
and set $s=1/t-1\in(0,1)$. 
Here we introduce a trace class operator 
$$
H_{s,X}u
=
\sum_{k=0}^{\infty}
(-s)^k
\sum_{\lvert\alpha\rvert=k}
\langle{u,u_\alpha}\rangle_{H_\Phi}\tilde{T}_{b(\cdot+X)}u_\alpha
$$
for $u{\in}H_\Phi$. 
Let $K_{s,X}(Y,Z)$ be the integral kernel of $H_{s,X}$, that is, 
$$
K_{s,X}(Y,Z)
=
\sum_{k=0}^\infty
(-s)^k
\sum_{\lvert\alpha\rvert=k}
\tilde{T}_{b(\cdot+X)}u_\alpha(Y)
\overline{u_\alpha(Z)}. 
$$
It is easy to see that 
$K_{s,X}(Y,Y){\in}L^1(\mathbb{C}^n;e^{-2\Phi(Y)/h}L(dY))$ 
since 
\begin{align*}
& \sum_{k=0}^\infty
  s^k
  \sum_{\lvert\alpha\rvert=k}
  \int_{\mathbb{C}^n}
  \lvert\tilde{T}_{b(\cdot+X)}u_\alpha(Y)\rvert
  \lvert\overline{u_\alpha(Y)}\rvert
  e^{-2\Phi(Y)/h}L(dY) 
\\
  \leqslant
& M\sum_{\alpha\in\mathbb{N}_0^n}s^{\lvert\alpha\rvert}
  =
  M\left(\sum_{k=0}^\infty{s^k}\right)^n
  =
  M(1-s)^{-n}
  =
  \frac{Mt^n}{(2t-1)^n}.
\end{align*}
Then, the Lebesgue convergence theorem and 
\eqref{equation:detroit} impliy that 
\begin{align}
& t^{-n}
  \int_{\mathbb{C}^n}
  \sum_{k=0}^N
  (-s)^k
  \sum_{\lvert\alpha\rvert=k}
  \tilde{T}_{b(\cdot+X)}u_\alpha(Y)
  \overline{u_\alpha(Y)}
  e^{-2\Phi(Y)/h}L(dY)
\nonumber
\\
  =
& \frac{C_\Phi}{(th)^n}
  \int_{\mathbb{C}^n}
  \sum_{k=0}^N
  (-s)^k
  \frac{1}{k!}
  \left(\frac{2\lvert{RY}\rvert^2}{h}\right)^k
  e^{-2\lvert{RY}\rvert^2/h}
  b(X+Y)
  L(dY)
\label{equation:police}
\end{align}
converges as $N\rightarrow\infty$. 
Thus we have \eqref{equation:sakurako} for $t\in(1/2,1)$ 
since the right hand side of \eqref{equation:police} 
converges to $b_t(X)$.  
\end{proof}
Next we prove Theorem~\ref{theorem:bounded}-(iii). 
Boulkhemair proved in \cite{boulkhemair} that 
\eqref{equation:sjoestrand} is equivalent to 
\begin{equation}
\sup_{X\in\mathbb{C}^n}
\lvert\mathscr{F}^{-1}[\mathscr{F}[b]\tau_\lambda\tilde{\chi}](X)\rvert
\in
L^1(\mathbb{C}^n_\lambda)
\label{equation:boulkhemair} 
\end{equation}
with some $\tilde{\chi}\in\mathscr{S}(\mathbb{C}^n)$ 
satisfying $\int_{\mathbb{C}^n}\tilde{\chi}(X)L(dX)\ne0$, 
where $\mathscr{F}^{-1}$ is the usual inverse Fourier transform on $\mathbb{C}^n$. 
\begin{proof}[Proof of Theorem~\ref{theorem:bounded}-{\rm (iii)}] 
We compute the condition \eqref{equation:boulkhemair}. 
We choose $\mathscr{F}[\chi](X)=C_1e^{-4\lvert{RX}\rvert^2/h}$ 
which is the heat kernel at the time $t=1/2$, 
and expect a comprehensive expression coming from the parallelogram law. 
Let $X^\ast\in\mathbb{C}^n$ be the dual variable under the Fourier transform. 
We choose a constant $C_1>0$ so that 
$\chi(X^\ast)=e^{-h\lvert{{}^t\bar{R}^{-1}X^\ast}\rvert^2/16}$. 
Set $\chi_\lambda=\tau_\lambda\chi$ for short. 
The parallelogram law implies that 
\begin{align*}
  \mathscr{F}[b_{1/2}](X^\ast)\chi_{2\bar{\lambda}}(X^\ast)
& =
  e^{-h\lvert{{}^t\bar{R}^{-1}X^\ast}\rvert^2/16
     -h\lvert{{}^t\bar{R}^{-1}(X^\ast-2\bar{\lambda})}\rvert^2/16}
  \mathscr{F}[b](X^\ast)
\\
& =
  e^{-h\lvert{{}^tR^{-1}\lambda}\rvert^2/8
     -h\lvert{{}^t\bar{R}^{-1}(X^\ast-\bar{\lambda})}\rvert^2/8}
  \mathscr{F}[b](X^\ast).
\end{align*}
Taking the inverse Fourier transformation of the above, we deduce 
\begin{align*}
& \mathscr{F}^{-1}[\mathscr{F}[b_{1/2}\chi_{2\bar{\lambda}}]](X)
\\
  =
& e^{-h\lvert{{}^tR^{-1}\lambda}\rvert^2/8}
  \frac{C_\Phi}{h^n}
  \int_{\mathbb{C}^n}
  e^{i\operatorname{Re}\hspace{1pt}\langle{X-Y,\lambda}\rangle
     -2\lvert{R(X-Y)}\rvert^2/h}
  b(Y)L(dY)
\\
  =
& e^{i\operatorname{Re}\hspace{1pt}\langle{X,\lambda}\rangle
     -h\lvert{{}^tR^{-1}\lambda}\rvert^2/8}
  (b^{-\lambda})_1(X).
\end{align*}
Hence, we obtain 
$$
\sup_{X\in\mathbb{C}^n}
\lvert\mathscr{F}^{-1}[\mathscr{F}[b_{1/2}]\chi_{-2\bar{\lambda}}](X)\rvert
=
e^{-h\lvert{{}^tR^{-1}\lambda}\rvert^2/8}
\lVert{(b^{\lambda})_1}\rVert_{L^\infty(\mathbb{C}^n)}. 
$$
This completes the proof.
\end{proof}
%
%
%%%%%%
%%%%%% Section 3
%%%%%%
\section{Deformation estimates for compositions}
\label{section:deformation}
Finally, we prove Theorem~\ref{theorem:deformation}. 
We first review the composition of pseudodifferential operators on $H_\Phi$. 
Let $\sigma$ be a canonical symplectic form on $\mathbb{C}^{2n}$, that is, 
$$
\sigma=d\Xi{\wedge}dX=\sum_{j=1}^nd\Xi_j{\wedge}dX_j
$$
at $(X,\Xi)\in\mathbb{C}^n\times\mathbb{C}^n$. 
Split $\sigma$ into real and imaginary parts, 
and denote $\sigma=\sigma_{\mathbb{R}}+i\sigma_I$. 
$\mathbb{R}^{2n}$ and $\Lambda_\Phi$ are 
$I$-Lagrangian and $\mathbb{R}$-symplectic. 
Indeed, this is obvious for $\mathbb{R}^{2n}$, and 
a direct computation shows that 
$\sigma_{I}\vert_{\Lambda_\Phi}=0$ and 
$$
\sigma_{\mathbb{R}}\vert_{\Lambda_\Phi}
=
2i
\sum_{j,k=1}^n
\frac{\partial^2\Phi}{\partial{X_j}\partial\bar{X}_k}
dX_j{\wedge}d\bar{X}_k
\quad\text{for}\quad
\theta=\frac{2}{i}\frac{\partial\Phi}{\partial{X}}(X),
$$
which is nondegenerate. 
We use this fact as 
$\kappa_T^\ast\sigma=\sigma_{\mathbb{R}}$ on $\mathbb{R}^{2n}$. 
\par
Let $a^\prime,b^\prime{\in}S_\text{W}(\Lambda_\Phi)$. 
It is well-known that 
$$
\operatorname{Op}_h^\text{W}(a^\prime\circ\kappa_T)
\circ
\operatorname{Op}_h^\text{W}(b^\prime\circ\kappa_T)
=
\operatorname{Op}_h^\text{W}(a^\prime\circ\kappa_T{\#}b^\prime\circ\kappa_T),
$$
\begin{align*}
  a^\prime\circ\kappa_T{\#}b^\prime\circ\kappa_T(x,\xi)
& = 
  \frac{1}{(2\pi{h})^{2n}}
  \int_{\mathbb{R}^{4n}}
  e^{-2i\sigma_{\mathbb{R}}(y,\eta;z,\zeta)/h}
\\
& \times
  a^\prime\circ\kappa_T(x+y,\xi+\eta)
  b^\prime\circ\kappa_T(x+z,\xi+\zeta)
  dyd{\eta}dzd{\zeta}.
\end{align*}
Set $\theta(X)=-2i\Phi^\prime_X(X)$ for $X\in\mathbb{C}^n$. 
Using the exact Egorov theorem \eqref{equation:egorov} 
together with the symplectic transform $\kappa_T$ 
or a direct computation, we have 
$$
\operatorname{Op}_h^\text{W}(a^\prime)
\circ
\operatorname{Op}_h^\text{W}(b^\prime)
=
\operatorname{Op}_h^\text{W}(a^\prime{\#}b^\prime),
$$
\begin{align*}
  a^\prime{\#}b^\prime(X,\theta(X))
& = 
  \left(\frac{2^nC_\Phi}{h^n}\right)^2
  \int_{\mathbb{C}^{2n}}
  e^{-2i\sigma(Y,\theta(Y);Z,\theta(Z))/h}
\\
& \times
  a^\prime(X+Y,\theta(X+Y))
  b^\prime(X+Z,\theta(X+Z))
  L(dY)L(dZ).
\end{align*}
\par
Here we begin the proof of Theorem~\ref{theorem:deformation}. 
Suppose that 
$\partial_X^\alpha\partial_{\bar{X}}^\beta{a}, 
\partial_X^\alpha\partial_{\bar{X}}^\beta{b} 
\in
S_\text{W}(\mathbb{C}^n)$ 
for any multi-indices satisfying 
$\lvert\alpha+\beta\rvert\leqslant3$. 
Set $a_t=e^{th\Delta}a$, $b_t=e^{th\Delta}b$, 
\begin{align*}
  a_{1/2}^\prime(X,\theta)
& =
  a_{1/2}
  \left(
  X,
  \frac{i}{2}(\Phi^{\prime\prime}_{X\bar{X}})^{-1}
  \left(\theta-\frac{2}{i}\Phi^{\prime\prime}_{XX}X\right)
  \right),
\\
  b_{1/2}^\prime(X,\theta)
& =
  b_{1/2}
  \left(
  X,
  \frac{i}{2}(\Phi^{\prime\prime}_{X\bar{X}})^{-1}
  \left(\theta-\frac{2}{i}\Phi^{\prime\prime}_{XX}X\right)
  \right).
\end{align*}
Then, we have 
$\tilde{T}_a\circ\tilde{T}_b=\operatorname{Op}_h^\text{W}(a_{1/2}^\prime{\#}b_{1/2}^\prime)$. 
Since $a_{1/2}^\prime(X,\theta(X))=a_{1/2}(X,\bar{X})$, 
if we write $a_t(X)=a_t(X,\bar{X})$ and $b_t(X)=b_t(X,\bar{X})$ simply, then 
$\tilde{T}_a\circ\tilde{T}_b
=\operatorname{Op}_h^\text{W}(a_{1/2}{\#}b_{1/2})$, and 
\begin{align*}
  a_t{\#}b_t(X)
& = 
  \left(\frac{2^nC_\Phi}{h^n}\right)^2
  \int_{\mathbb{C}^{2n}}
  e^{-2i\sigma(Y,\theta(Y);Z,\theta(Z))/h}
\\
& \times
  a_t(X+Y)
  b_t(X+Z)
  L(dY)L(dZ).
\end{align*}
To complete the proof of Theorem~\ref{theorem:deformation}, 
we have only to show that 
\begin{equation}
a_{1/2}{\#}b_{1/2}
\equiv
e^{h\Delta/2}(ab)-\frac{h}{2}e^{h\Delta/2}Q(a,b)
\quad\text{mod}\quad
h^2S_\text{W}(\mathbb{C}^n).
\label{equation:kiyomi} 
\end{equation}
Here we remark that 
$$
-2i\sigma(Y,\theta(Y);Z,\theta(Z))
=
4\langle{Y,\Phi^{\prime\prime}_{X\bar{X}}\bar{Z}}\rangle
-
4\langle{Z,\Phi^{\prime\prime}_{X\bar{X}}\bar{Y}}\rangle
=
8i\operatorname{Im}
\langle{Y,\Phi^{\prime\prime}_{X\bar{X}}\bar{Z}}\rangle,
$$
\begin{align*}
  Ye^{-2i\sigma(Y,\theta(Y);Z,\theta(Z))/h}
& =
  \frac{h}{4}
  (\Phi^{\prime\prime}_{\bar{X}X})^{-1}
  \frac{\partial}{\partial\bar{Z}}
  e^{-2i\sigma(Y,\theta(Y);Z,\theta(Z))/h},
\\
  \bar{Y}e^{-2i\sigma(Y,\theta(Y);Z,\theta(Z))/h}
& =
  -
  \frac{h}{4}
  (\Phi^{\prime\prime}_{X\bar{X}})^{-1}
  \frac{\partial}{\partial{Z}}
  e^{-2i\sigma(Y,\theta(Y);Z,\theta(Z))/h}.
\end{align*}
From Taylor's formula and the integration by parts we derive 
$$
a_t{\#}b_t(X)
=
a_tb_t(X)
-
\frac{h}{4}Q(a_t,b_t)(X)
+
\frac{h}{4}Q(b_t,a_t)(X)
+
h^2r_t(X;h),
$$
where $\{r_t(X;h)\}_{h\in(0,1]}$ is bounded in
$\mathscr{B}^\infty(\mathbb{C}^n)$ 
for fixed $t>0$. 
\par
We approximate the main term of $a_t{\#}b_t$ which is 
$$
c_t
=
a_tb_t
-
\frac{h}{4}Q(a_t,b_t)
+
\frac{h}{4}Q(b_t,a_t),
$$
by constructing an approximate solution to the initial value problem for
the heat equation satisfied by $c_t$. 
In other words, we construct an asymptotic solution to the transport
equation whose main term is given by the heat operator $\partial_t-h\Delta$. 
It is easy to see that 
$$
\partial_X^\alpha\partial_{\bar{X}}^\beta{a_t}, 
\partial_X^\alpha\partial_{\bar{X}}^\beta{b_t} 
\in
C([0,\infty);S_\text{W}(\mathbb{C}^n))
$$
for $\lvert\alpha+\beta\rvert\leqslant3$.  
Set $p_t=e^{th\Delta}(ab)+hp^{(1)}_t$ and 
\begin{align*}
  p^{(1)}_t
& =
  -
  \frac{1}{4}e^{th\Delta}Q(a,b)
  +
  \frac{1}{4}e^{th\Delta}Q(b,a)
\\
& -
  \frac{1}{2}
  \int_0^t
  e^{(t-s)h\Delta}
  \{Q(a_s,b_s)+Q(b_s,a_s)\}
  ds.
\end{align*}
Then, $c_t$ and $p_t$ solve 
\begin{align*}
  \left(\frac{\partial}{\partial{t}}-h\Delta\right)c_t
& =
  -
  \frac{h}{2}
  \{Q(a_t,b_t)+Q(b_t,a_t)\}
  +
  \frac{h^2}{4}
  Q_1(a_t,b_t),
\\
  c_0
& =
  ab-\frac{h}{4}Q(a,b)+\frac{h}{4}Q(b,a),  
\\
  Q_1(a,b)
& =
  \left\langle
  (\Phi_{X\bar{X}}^{\prime\prime})^{-1}
  \frac{\partial^2a}{\partial{X}^2},
  (\Phi_{\bar{X}X}^{\prime\prime})^{-1}
  \frac{\partial^2b}{\partial\bar{X}^2}
  \right\rangle,
\\
  \left(\frac{\partial}{\partial{t}}-h\Delta\right)p_t
& =
  -
  \frac{h}{2}
  \{Q(a_t,b_t)+Q(b_t,a_t)\},
\\
  p_0
& =
  ab-\frac{h}{4}Q(a,b)+\frac{h}{4}Q(b,a),  
\end{align*}
respectively. 
Hence, 
$$
c_t-p_t
=
\frac{h^2}{4}
\int_0^t
e^{(t-s)h\Delta}
Q_1(a_s,b_s)
ds
\in
h^2C([0,\infty);S_\text{W}(\mathbb{C}^n)).
$$
\par
We show that the main part of the second term in $p^{(1)}_t$ 
is $-te^{th\Delta}\{Q(a,b)+Q(b,a)\}/2$, that is,  
$$
\int_0^te^{(t-s)h\Delta}\{Q(a_s,b_s)+Q(b_s,a_s)\}ds
=
te^{th\Delta}\{Q(a,b)+Q(b,a)\}
+
\mathcal{O}(h).
$$
For this purpose, we estimate 
\begin{align*}
  \int_0^te^{(t-s)h\Delta}Q(a_s,b_s)ds
  -
  te^{th\Delta}Q(a,b)
& =
  \int_0^t
  \{e^{(t-s)h\Delta}Q(a_s,b_s)-Q(a_s,b_s)\}
  ds
\\
& +
  \int_0^t
  \{Q(a_s,b_s)-Q(a_t,b_t)\}
  ds
\\
& +
  t\{Q(a_t,b_t)-e^{th\Delta}Q(a,b)\}
\\
& =
  F_t+G_t+tH_t.
\end{align*}
\par
We here remark that the heat kernel $e^{th\Delta}$ is an even function
in the space variable. 
Combining this fact and Taylor's formula, we can obtain the desired
estimates of $F_t$ and $G_t$. 
This technique has been frequently used for approximating symbols. 
Changing the variables in the explicit formula of the heat kernel, 
we have 
\begin{align}
  F_t(X)
& =
  \int_0^t
  \frac{C_\Phi}{\{(t-s)h\}^n}
  ds
  \int_{\mathbb{C}^n}
  e^{-2\lvert{RY}\rvert^2/(t-s)h}
\nonumber
\\
& \times
  \{Q(a_s,b_s)(X+Y)-Q(a_s,b_s)(X)\}
  L(dy)
\nonumber
\\
& =
  C_\Phi\int_0^tds
  \int_{\mathbb{C}^n}
  e^{-2\lvert{RY}\rvert^2}
\nonumber
\\
& \times
  \{Q(a_s,b_s)(X+\sqrt{(t-s)h}Y)-Q(a_s,b_s)(X)\}
  L(dy).
\label{equation:maki}
\end{align}
Substituting Taylor's formula 
\begin{align*}
  Q(a_s,b_s)(X+Y)
& =
  Q(a_s,b_s)(X)
  +
  \langle{Y},\partial_XQ(a_s,b_s)(X)\rangle
\\
& +
  \langle\bar{Y},\partial_{\bar{X}}Q(a_s,b_s)(X)\rangle
  +
  Q_2(a_s,b_s)(X,Y),
\end{align*}
$$
Q_2(a_s,b_s)(X,Y)
=
\sum_{\lvert\alpha+\beta\rvert=2}
\frac{Y^\alpha\bar{Y}^\beta}{\alpha!\beta!}
\int_0^1(1-\tau)
\left(\frac{\partial^2Q(a_s,b_s)}{\partial{X^\alpha}\partial\bar{X}^\beta}\right)
(X+\tau{Y})d\tau
$$
into \eqref{equation:maki}, we have 
$$
F_t(X)
=
C_\Phi
\int_0^tds
\int_{\mathbb{C}^n}
e^{-2\lvert{RY}\rvert^2}
Q_2(a_s,b_s)(X,\sqrt{(t-s)h}Y)
L(dY),
$$
which belongs to $hC([0,\infty);S_\text{W}(\mathbb{C}^n))$. 
\par
We split $G_t$ into two parts
$$
G_t
=
\int_0^t
\{Q(a_s,b_s)-Q(a_t,b_t)\}
ds
=
\int_0^t
\{Q(a_s-a_t,b_s)+Q(a_t,b_s-b_t)\}
ds.
$$
Since 
\begin{align*}
  a_s(X)-a_t(X)
& =
  C_\Phi
  \int_{\mathbb{C}^n}
  e^{-2\lvert{RY}\rvert^2}
  \{a(X+\sqrt{sh}Y)-a(X+\sqrt{th}Y)\}
  L(dY)
\\
& =
  C_\Phi
  \int_{\mathbb{C}^n}
  e^{-2\lvert{RY}\rvert^2}
  \tilde{a}(X,(\sqrt{s}-\sqrt{t})\sqrt{h}Y)
  L(dY),
\end{align*}
$$
\tilde{a}(X,Y)
=
\sum_{\lvert\alpha+\beta\rvert=2}
\frac{Y^\alpha\bar{Y}^\beta}{\alpha!\beta!}
\int_0^1(1-\tau)
\left(\frac{\partial^2a}{\partial{X^\alpha}\partial\bar{X}^\beta}\right)
(X+\tau{Y})d\tau,
$$
we can show that $G_t{\in}hC([0,\infty);S_\text{W}(\mathbb{C}^n))$. 
\par
It follows that 
$H_t{\in}hC([0,\infty);S_\text{W}(\mathbb{C}^n))$ 
since 
$$
\left(\frac{\partial}{\partial{t}}-h\Delta\right)H_t
\in
hC([0,\infty);S_\text{W}(\mathbb{C}^n)), 
\quad
H_0=0.
$$
\par
Combining the estimates of $F_t$, $G_t$ and $H_t$, 
we have 
\begin{equation}
\int_0^t
e^{(t-s)h\Delta}
Q(a_s,b_s)ds
-
te^{th\Delta}Q(a,b)
\in
hC([0,\infty);S_\text{W}(\mathbb{C}^n)). 
\label{equation:todo} 
\end{equation}
Applying \eqref{equation:todo} to $p^{(1)}_t$, we obtain 
$$
c_t
=
e^{th\Delta}(ab)
-
\frac{h}{2}
\left(\frac{1}{2}+t\right)
e^{th\Delta}Q(a,b)
+
\frac{h}{2}
\left(\frac{1}{2}-t\right)
e^{th\Delta}Q(b,a)
+
\mathcal{O}(h^2).
$$
If we take $t=1/2$, we obtain \eqref{equation:kiyomi}. 
This completes the proof of Theorem~\ref{theorem:deformation}. 
\vspace{11pt}\\
{\bf ACKNOWLEDGEMENT}
\quad 
The author would like to express to the referee 
his sincere gratitude for valuable comments. 
In particular, the present paper improved in the presentation 
following the referee's suggestion on the logic in the first section. 
%%%%%%
%%%%%% References
%%%%%%

%%%%%%
%%%%%% End
%%%%%%
\end{document}